\documentclass[11pt]{amsart}
\usepackage{amsmath,amssymb,amsfonts,amsthm,amscd,indentfirst}

\usepackage{hyperref}\usepackage{xcolor}

\newtheorem{prop}{Proposition}[section]

\newtheorem{coro}[prop]{Corollary}

\def\and{\quad{\rm and}\quad}

\def\<{\langle}
\def\>{\rangle}

\title[Interior $C^2$ estimate ]{Interior $C^2$ regularity of convex solutions to prescribing scalar curvature equations}
\author{Pengfei Guan
        \and
        Guohuan Qiu
        }
\address{Department of Mathematics and Statistics, McGill University, 805 Sherbrooke O, Montreal, Quebec, Canada, H3A 0B9}
\email{pengfei.guan@mcgill.ca}
\email{guohuan.qiu@mail.mcgill.ca}
\thanks{Research of the first author was supported in part by an NSERC Discovery Grant. }

\usepackage{graphicx}
\newtheorem{theorem}{Theorem}
\newtheorem{lemma}{Lemma}

\begin{document}
\begin{abstract}
We establish interior $C^2$ estimates for convex solutions of scalar curvature equation and $\sigma_2$-Hessian equation. We also prove interior curvature estimate for isometrically immersed hypersurfaces $(M^n,g)\subset \mathbb R^{n+1}$ with positive scalar curvature.  These estimates are consequences of an interior estimates for these equations obtained under a weakened condition.
\end{abstract}
\date{\today}
\maketitle

\section{Introduction}

Regularity estimates for immersed convex surfaces in $\mathbb R^3$ is the key to solutions of the Weyl problem by Nirenberg and Pogorelov \cite{N, Pb}.   A refined interior estimate for solutions to the Weyl problem by Heinz  \cite{H} reveals some special properties of solutions to Monge-Amp\`ere type equations in dimension $2$. This type of interior estimates are important for existence of isometric embedding of non-compact surfaces and for Liouville type theorems.
An interesting question in geometric analysis is if such purely local interior curvature estimates hold for isometrically immersed hypersurfaces in $\mathbb R^{n+1}$ for $n\ge 2$.  We provide an affirmative answer in this paper. The following is a generalization Heinz's interior estimate \cite{H} in higher dimensions for the isometrically embedded hypersurfaces.

\begin{theorem}\label{thm01} Suppose $(M^n, g)$ is an isometrically immersed hypersurface in $\mathbb R^{n+1}$ with positive scalar curvature $R_g$. Suppose $M^n$ is a $C^1$ graph over a ball $B_r\subset \mathbb R^n$ of radius $r$. Then there is constant $C$ depending only on $n, r, \|g\|_{C^4(B_r)}, \inf_{B_r}R_g, \|M\|_{C^1(B_r)}$ such that
\begin{equation}\label{curvest} \max_{B_{\frac{r}2}}|\kappa_i|\le C,\end{equation}
where $\kappa_1,\cdots,\kappa_n$ are the principal curvatures of $M$.
 \end{theorem}

 \medskip
 
The above result is related to a longstanding problem in fully nonlinear partial differential equations:
the interior $C^2$ estimate for solutions of the following prescribing scalar curvature equation  and $\sigma_2$-Hessian equation,
\begin{equation} \label{seq:maineq}
\sigma_2(\kappa_1(x),\cdots,\kappa_n(x))=f(X,\nu(x))>0,\quad X\in B_r\times \mathbb R \subset \mathbb R^{n+1} \end{equation}
and
\begin{equation}
\sigma_{2}(\nabla^{2}u(x))=f(x,u(x),\nabla u(x))>0, \quad x\in B_r\subset \mathbb R^n\label{eq:maineq}
\end{equation}where $\kappa_1, \cdots, \kappa_n$ are the principal curvatures and $\nu$ the normal of the given hypersurface as a gragh over a ball $B_r\subset \mathbb R^n$ respectively,  $\sigma_{k}$ the $k$-th elementary symmetric
function, $ 1\le k\le n$.

Equations (\ref{seq:maineq}) and (\ref{eq:maineq}) are special cases of $\sigma_{k}$-Hessian and curvature equations developed by Caffarealli-Nirenberg-Spruck in \cite{CNS3}, as an integrated part of fully nonlinear
PDE. A $C^2$ function $u$ is called an admissible solution to equation (\ref{eq:maineq}) if $u$ satisfies the equation and $\Delta u>0$. Similarly, if a graph $(x,u(x))$ is called an admissible solution to equation (\ref{seq:maineq}) if $u$ satisfies the equation with positive mean curvature.

The global regularity of solutions to Dirichlet boundary problem
was established in \cite{CNS3}. When $n=2$, these equation is the Monge-Amp\`ere
equation type. Interior $C^2$ estimate was proved by Heinz \cite{H}. Such interior estimate fails for general convex solutions of Monge-Amp\`ere equation in higher dimensions,
\[
\sigma_{n}(\nabla^{2}u(x))=1,\quad x\in B_{1}\subset\mathbb{R}^{n},
\]
when $n\ge3$,   there are counter-examples constructed by Pogorelov in \cite{P}. Pogorelov's counter-examples were extended by Urbas \cite{U} to general $\sigma_{k}$ Hessian and curvature equations
\[\sigma_k(\kappa_1,\cdots,\kappa_n)=f,\quad
\sigma_{k}(\nabla^{2}u)=f,  \quad \quad k\ge3.\]
Whether interior $C^{2}$ estimates for solutions of equations \eqref{seq:maineq} and \eqref{eq:maineq} when $n\ge 3$ reminds open in general.
A major progress was achieved by Warren-Yuan \cite{WY}, they obtained
$C^{2}$ interior estimate in the case $n=3$ of equation
\begin{equation}\label{sigma21}
\sigma_{2}(\nabla^{2}u)=1,\quad x\in B_{1}\subset \mathbb R^n.
\end{equation}
More recently, the interior $C^2$ estimate for {\it semi-convex} solutions of equation (\ref{sigma21}) is obtained by McGonagle-Song-Yuan \cite{MSY}

\medskip

We establish interior estimates for {\it convex} solutions to equations (\ref{seq:maineq}) and (\ref{eq:maineq}) in higher dimensions. 

\begin{theorem}\label{thm1sv} Suppose $M$ is a convex graph over $B_r\subset \mathbb R^n$ and it is
a solution of equation \eqref{seq:maineq}, then
\begin{equation}
\max_{x\in B_{\frac{r}{2}}}|\kappa_i(x)|\leq C,\label{sestHc}
\end{equation}
where constant $C$ depending only on $n$, $r$, $\|M\|_{C^{1}(B_r)}$,
$\|f\|_{C^{2}(B_r)}$ and $\|\frac{1}{f}\|_{L^{\infty}}(B_r)$.  \end{theorem}

The similar interior $C^2$ estimate also holds for solutions of $\sigma_2$-Hessian equation (\ref{eq:maineq}).
\begin{theorem}\label{thm1v} Suppose $u\in C^{4}(\bar{B}_{1})$ is
a convex solution of equation \eqref{eq:maineq}, then
\begin{equation}
\max_{x\in B_{\frac{1}{2}}}|\nabla^{2}u(x)|\leq C,\label{estHc}
\end{equation}
where constant $C$ depending only on $n$, $\|u\|_{C^{1}}(B_1)$,
$\|f\|_{C^{2}}(B_1)$ and $\|\frac{1}{f}\|_{L^{\infty}}(B_1)$.  \end{theorem}

In the case
of $n=2$, the above results were proved by Heinz \cite{H} (see also \cite{CHO}). One observes that equations \eqref{eq:maineq} and \eqref{seq:maineq} are Monge-Amp\`ere type equation, all admissible solutions are automatically convex. If $n>2$, admissible solutions of these equations in general are not convex. Under additional assumption that $\sigma_3>-A$ for some constant $A\ge 0$, we will obtain $C^2$ interior estimates for solutions of equations \eqref{eq:maineq} and \eqref{seq:maineq} in Theorem \ref{thm1s} and Theorem \ref{thm1}. Theorem \ref{thm1sv} and Theorem \ref{thm1v} will be the direct consequence of these theorems respectively.  We note that interior $C^2$ estimates for equations (\ref{eq:maineq}) and (\ref{seq:maineq}) trivially yield $\sigma_3>-A$. It is an open question if this assumption is redundant. 
This assumption is {\it not needed} in Theorem \ref{thm01}  due to the control of Ricci curvature from intrinsic geometry. In dimension three case, the second name author can also prove Theorem \ref{thm1v} without this assumption in [13].

\bigskip

The organization of the paper as follow. We collect and prove some facts associated to the second elementary symmetric function in section 2. Theorem \ref{thm1} will be proved in section 3, and Theorem \ref{thm1s} will be proved in section 4. The main theorems in the Introduction follow from these theorems.

\section{Preliminaries}

Let $W=(W_{ij})$ be a symmetric tensor, we say $W\in \Gamma_2$ if $\sigma_1(W)>0, \sigma_2(W)>0$.
Suppose $u\in C^2(B_1)$, $u$ is called an admissible solution of equation (\ref{eq:maineq}) if $\nabla^2 u(x)\in \Gamma_2, \forall x\in B_1$. Likewise, if $h$ is the second fundamental form of the hypersurface $M$, we say $M$ is an admissible solution of equation (\ref{seq:maineq}) if $h(x)\in \Gamma_2, \forall x\in M$.


It follows from \cite{CNS3}, if $W\in \Gamma_2$, then $\sigma_2^{ij}(W)=\frac{\partial \sigma_2}{\partial W_{ij}}(W)$ is positive definite. We list some known facts regarding structure of $\sigma_2$. The first lemma is from \cite{LinTru}.
\begin{lemma}\label{lemma1} Suppose  $W\in {\Gamma}_{2}$ is diagonal and $W_{11}\ge\cdots\ge W_{nn}$, then there exist $c_{1}>0$ and $c_{2} >0$ depending
only on $n$ such that
\begin{equation}\label{tildec0}
\sigma_{2}^{11}(W)W_{11}\geq c_{1}\sigma_2(W),
\end{equation}
and for any $j\geq2$
\begin{equation}\label{s222}
\sigma_{2}^{jj}(W)\geq c_{2}\sigma_{1}(W).
\end{equation}
\end{lemma}

The following lemma can be found in \cite{chen2013optimal}.
\begin{lemma}\label{chenlem} Under the same assumption as in Lemma \ref{lemma1},  if $\xi_{ij}$ is symmetric  and
\begin{eqnarray*}
\sum_{i=2}^{n}\sigma_{2}^{ii}\xi_{ii}+\sigma_{2}^{ii}\xi_{11}=\eta,
\end{eqnarray*}
then
\begin{equation}\label{chen}
-\sum_{i\neq j}\xi_{ii}\xi_{jj}\geq\frac{1}{2\sigma_{2}(W)}\frac{(n-1)[2\sigma_{2}(W)\xi_{11}-W_{11}\eta]^{2}}{[(n-1)W_{11}^{2}
+2(n-2)\sigma_{2}(W)]}-\frac{\eta^{2}}{2\sigma_{2}(W)}.
\end{equation}
\end{lemma}

\begin{lemma}\label{GQlem} Under the same assumption as in Lemma \ref{lemma1}, and in addition that there exist a positive
constant $a\le \sqrt{\frac{\sigma_2(W)}{3(n-1)(n-2)}}$ (if $n=2$, $a>0$ could be arbitrary), such that
\begin{equation}\label{assumption}
\sigma_{3}(W+aI)(x_{0})\geq0,
\end{equation}
then \begin{equation}\label{acon}
\frac76\sigma_2(W)\geq(f+\frac{(n-1)(n-2)}{2}a^{2})\geq\frac56\sigma_{2}^{11}(W)W_{11},
\end{equation}
provided that $W_{11}> 6(n-2)a$.
Furthermore, for any $j\in \{2,\cdots n\}$,
\begin{equation}\label{eq:assum}
|W_{jj}|\leq(n-1)^{2}a
+\frac{7(n-1)\sigma_2(W)}{5W_{11}}.
\end{equation}
\end{lemma}

\begin{proof}
As $W+aI\in \bar \Gamma_3$, it follows that $(W_{22}+a,\cdots, W_{nn}+a)\in \bar \Gamma_2$, thus
\begin{eqnarray*}
\sigma_{2}(W)-W_{11}\sigma_{2}^{11}(W)+a(n-2)\sigma_{2}^{11}(W)+a^{2}\frac{(n-1)(n-2)}{2}\geq0.\label{eq:acon0}
\end{eqnarray*}

Since $\frac{(n-1)(n-2)}{2}a^{2}\leq\frac{\sigma_2(W)}{6}$ and we may assume $W_{11}-a(n-2)\geq \frac{5W_{11}}{6}$.
Therefore
\begin{equation*}
\frac{7}{6}\sigma_2(W)\geq(\sigma_2(W)+\frac{(n-1)(n-2)}{2}a^{2})\geq(W_{11}-a(n-2))\sigma_{2}^{11}(W)
\geq\frac{5W_{11}}{6}\sigma_{2}^{11}(W).
\end{equation*}

As $W$ is diagonal, $W_{11}\geq W_{22}\geq\cdots\geq W_{nn}$,
and $W+aI\in \bar\Gamma_{3}$,
thus $\sum_{j\ge3}(W_{jj}+a)\ge0$. We have
\begin{eqnarray*}
\sigma_{2}^{11}(W+aI)=W_{22}+a+\sum_{j\ge3}(W_{jj}+a)\geq W_{22}+a,
\end{eqnarray*}
and
\begin{eqnarray*}
W_{22}\leq(n-2)a+\sigma_{2}^{11}(W).\label{eq:assump}
\end{eqnarray*}

On the other hand, as $W\in\Gamma_{2}$,
\begin{eqnarray*}
0\leq\sigma_{2}^{11}(W)\leq(n-2)W_{22}+W_{nn},
\end{eqnarray*}
and
\begin{eqnarray*}
-W_{jj}\leq-W_{nn} & \leq(n-2)W_{22}.\label{eq:ele1}
\end{eqnarray*}

For any $j$ from $2$ to $n$, we obtain \eqref{acon},
\begin{eqnarray*}
|W_{jj}|\leq(n-1)W_{22}\leq(n-1)^{2}a+(n-1)\sigma_{2}^{11}(W)\leq(n-1)^{2}a
+\frac{7(n-1)\sigma_2(W)}{5W_{11}}.
\end{eqnarray*}
\end{proof}

\begin{coro}\label{GQlemW} Under the same assumption as in Lemma \ref{lemma1}, suppose there is $A>0$ such that
\begin{equation}\label{CA3} \sigma_3(W)\ge -A.\end{equation}
If $W^{\frac32}_{11}>6(n-2)\sqrt{A}$, then \begin{equation}\label{aconW} \frac75\sigma_2(W)\geq\sigma_{2}^{11}(W)W_{11},
\end{equation}
and
there is constant $C$ depending only on $n, A, \sigma_2(W)$ such that, for any $j\in \{2,\cdots n\}$,
\begin{equation}\label{eq:assumW}
|W_{jj}|\leq C W^{-\frac12}_{11}.
\end{equation}
\end{coro}
\begin{proof} One may pick $a^2=\frac{A}{\sigma_1(W)}$, then condition (\ref{assumption}) is satisfied. The corollary follows from Lemma \ref{GQlem}.  \end{proof}

\begin{coro}\label{GQlemR} Suppose $M^n$ is an immersed hypersurface in $\mathbb R^{n+1}$, suppose at a point $x_0\in M$, the scalar curvature of $M$ at $x_0$ is positive and $\kappa_1\ge \kappa_2\ge \cdots \ge \kappa_n$ are the principal curvatures of $M$ and $\sigma_1(\kappa_1,\cdots,\kappa_n)>0$ at $x_0$. Suppose Ricci curvature is either bounded below by $-A$ or bounded above by $A$ for some constant $A$, then there is constant $C$ depending only on $n, A, \sigma_2(\kappa_1,\cdots,\kappa_n)$ such that, for any $j\in \{2,\cdots n\}$,
\begin{equation}\label{eq:assumR}
|\kappa_{j}|\leq C \kappa^{-1}_{1}.
\end{equation}
If $\kappa_1$ is sufficiently large compared to $A$,  then \begin{equation}\label{aconR} \frac75\sigma_2(\kappa_1,\cdots,\kappa_n)\geq\sigma_{2}^{11}(\kappa_1,\cdots,\kappa_n)\kappa_{1}.
\end{equation} \end{coro}
\begin{proof} Estimate (\ref{eq:assumR}) was proved in \cite{GLX} and \cite{GLu}. Here we use Lemma \ref{lemma1}. By the assumption, $(\kappa_1,\cdots,\kappa_n)\in \Gamma_2$. If $Ric_{ii}\ge -A, \forall i$, then \[Ric_{jj}\le (n-1)A+\sigma_2(\kappa_1,\cdots,\kappa_n), \forall j,\] since the scalar curvature is $\sigma_2(\kappa_1,\cdots,\kappa_n)$. Thus, Ricci is bounded from below and above. For $j\ge 2$, $\sigma_2^{jj}\kappa_j=Ric_{jj}$. By Lemma \ref{lemma1}, $\sigma_2^{jj}\ge c\kappa_1$, we get (\ref{eq:assumR}).

Note that $\sigma_3(\kappa_1,\cdots,\kappa_n)=\sum_{i<j<l} \kappa_i\kappa_j\kappa_l$. By (\ref{eq:assumR}), \[\sigma_3(\kappa_1,\cdots,\kappa_n)\ge -\tilde A,\]
where $\tilde A$ depends only on $A, n, \sigma_2$. Therefore, (\ref{aconR}) follows from (\ref{aconW}).
 \end{proof}

\section{Hessian Equation}

In this section, we prove the following interior $C^2$ estimate for Hessian equation
 \begin{theorem}\label{thm1} Suppose $u\in C^{4}(\bar{B}_{1})$ is
a solution of equation \eqref{eq:maineq}
and assume that there is a nonnegative constant
$A$ such that
\begin{equation}\label{aA} \sigma_{3}(\nabla^{2}u(x))\ge-A, \forall x\in B_1,\end{equation} then
\begin{equation}
\max_{x\in B_{\frac{1}{2}}}|\nabla^{2}u(x)|\leq C,\label{estH}
\end{equation}
where constant $C$ depending only on $n$, $A$, $\|u\|_{C^{1}}(B_1)$,
$\|f\|_{C^{2}}(B_1)$ and $\|\frac{1}{f}\|_{L^{\infty}}(B_1)$. In particular, if $u$ is convex, interior estimate (\ref{estH}) holds. \end{theorem}

\noindent
{\it Proof of Theorem \ref{thm1}.}
To establish the interior
estimate, for any $x\in\bar{B_{1}}$ and $\beta\in\mathbb{S}^{n-1}$,
we would like to obtain an upper bound of function for $x\in B_1, \vartheta\in \mathbb S^{n-1}$,
\begin{eqnarray*}
\tilde P(x,\vartheta)=2\log\rho(x)+\alpha(\frac{|\nabla u|^{2}}{2})+\beta(x\cdot\nabla u-u)+\log\log\max\{u_{\vartheta\vartheta},2\},
\end{eqnarray*}
where $\rho(x)=1-|x|^{2}$, and $\alpha,\;\beta$ are constants to be determined
later. The pick of test function $P$ is inspired by a recent work
\cite{GRW}.

The maximum value of $\tilde P(x,\beta)$ in $\bar{B_{1}}\times\mathbb{S}^{n-1}$
must be attained in an interior point of $B_{1}$ since $\rho=0$
on the boundary. We suppose $x_{0}\in B_{1}$ is a maximum point,
and maximum direction $\vartheta(x_{0})=e_{1}$. Then it can be seen easily
that $u_{1i}(x_{0})=0$ for any $i=2,\cdots n$. So we may assume
$\nabla^{2}u$ is diagonalized at this point. We choose coordinate
frame $\{e_{1},e_{2},\cdots e_{n}\}$, such that $\nabla^{2}u(x_{0})$
is diagonal and $u_{11}(x_{0})\ge u_{22}(x_{0})\ge\cdots\ge u_{nn}(x_{0})$.
We may
assume that $u_{11}(x_{0})\ge3$ is sufficiently large.
Now consider function
\begin{eqnarray}
 P(x)=2\log\rho(x)+\alpha(\frac{|\nabla u|^{2}}{2})+\beta(x\cdot\nabla u-u)+\log\log u_{11}.
\end{eqnarray}
Note that $x_0$ is also a maximum point of $P$.
We now want to estimate $P(x_{0})$.

Denote $b=\log u_{11}$.  At $x_{0}$,
\begin{equation}
0=P_{i}=\frac{2\rho_{i}}{\rho}+\alpha u_{k}u_{ki}+\beta(x_{k}u_{ki})+\frac{b_{i}}{b},\label{eq:fp}
\end{equation}
\begin{eqnarray*}
\quad\quad P_{ij}=2\frac{\rho_{ij}}{\rho}-2\frac{\rho_{i}\rho_{j}}{\rho^{2}}+\alpha(u_{kj}u_{ki}+u_{k}u_{kij})+\beta(u_{ij}+x_{k}u_{kij})+\frac{b_{ij}}{b}-\frac{b_{i}b_{j}}{b^{2}}.\label{eq:fp2}
\end{eqnarray*}

Contracting with $\sigma_{2}^{ij}:=\frac{\partial\sigma_{2}(\nabla^{2}u)}{\partial u_{ij}}$,
\begin{eqnarray*}
\sigma_{2}^{ij}P_{ij} & = & -4\frac{\sum\sigma_{2}^{ii}}{\rho}-8\frac{\sigma_{2}^{ii}x_{i}^{2}}{\rho^{2}}+\alpha\sigma_{2}^{ii}u_{ii}^{2}+\alpha\sigma_{2}^{ii}u_{iik}u_{k}\nonumber \\
 &  & +\beta\sigma_{2}^{ii}u_{ii}+\beta x_{k}\sigma_{2}^{ii}u_{iik}+\frac{\sigma_{2}^{ii}b_{ij}}{b}-\frac{\sigma_{2}^{ii}b_{i}^{2}}{b^{2}}.\label{eq:}
\end{eqnarray*}

Differentiate  equation \eqref{eq:maineq} in $k$-th variable,
\begin{eqnarray*}
\sigma_{2}^{ii}u_{iik}=\tilde{f}_{k},
\end{eqnarray*}
where $\tilde{f}(x)=f(x,u(x),\nabla u(x))$. 
Thus,
\begin{eqnarray}
\sigma_{2}^{ij}P_{ij} & = & -4\frac{(n-1)\sigma_{1}}{\rho}-8\frac{\sigma_{2}^{ii}x_{i}^{2}}{\rho^{2}}+\alpha(f\sigma_{1}-3\sigma_{3})
+2\beta f\nonumber \\
 &  & +\alpha\tilde{f}_{k}u_{k}+\beta\tilde{f}_{k}x_{k}+\frac{\sigma_{2}^{ii}b_{ii}}{b}
 -\frac{\sigma_{2}^{ii}b_{i}^{2}}{b^{2}}.\label{eq:P1}
\end{eqnarray}

\medskip{}

First we deal with the the last two terms in (\ref{eq:P1}).
Recall $b=\log u_{11}$. We have
\begin{equation}
b_{i}=\frac{u_{11i}}{u_{11}},
\end{equation}

and
\begin{equation}
\sigma_{2}^{ii}b_{ii}=\frac{\sigma_{2}^{ii}u_{11ii}}{u_{11}}-\frac{\sigma_{2}^{ii}u_{11i}^{2}}{u_{11}^{2}}.\label{eq:b2-1}
\end{equation}

Take one more derivative of the equation,
\begin{eqnarray*}
\sigma_{2}^{ij}u_{ij11}=\sum_{i\neq j}u_{ij1}^{2}-\sum_{i\neq j}u_{ii1}u_{jj1}+\tilde{f}_{11}.
\end{eqnarray*}
Insert it to \eqref{eq:b2-1},
\begin{eqnarray*}
\sigma_{2}^{ii}b_{ii}\geq\frac{\sum_{i\neq j}u_{ij1}^{2}-\sum_{i\neq j}u_{ii1}u_{jj1}+\tilde{f}_{11}}{u_{11}}-\frac{\sigma_{2}^{ii}u_{11i}^{2}}{u_{11}^{2}}.
\end{eqnarray*}

Set $W=\nabla^2 u$, $\xi_{ij}=u_{ij1}$ and $\eta=\tilde{f}_1$, in view of Lemma \ref{chenlem} and
the Cauchy-Schwarz inequality, we have $\forall n\ge2$,
\begin{equation}
-\sum_{i\neq j}u_{ii1}u_{jj1}\geq\frac{1.5fu_{111}^{2}}{u_{11}^{2}}-Cu_{11}^{2},\label{chen1}
\end{equation}
where $C$ depending on $\|f\|_{C^{1}}$, $\|\frac{1}{f}\|_{L^{\infty}}$
and $\|u\|_{C^{1}}$. In the rest of this paper, we will denote
$C$ to be constant under control (depending only on $n$, $\|f\|_{C^{2}}$,
$\|\frac{1}{f}\|_{L^{\infty}}$ and $\|u\|_{C^{1}}$), which may change
line by line.

It follows that
\begin{eqnarray}
\sigma_{2}^{ii}b_{ii} & \geq & \frac{2\sum_{i\geq2}u_{11i}^{2}}{u_{11}}+\frac{1.5fu_{111}^{2}}{u_{11}^{3}}
-\frac{\sigma_{2}^{ii}u_{11i}^{2}}{u_{11}^{2}}+\frac{f_{p_{i}}u_{i11}}{u_{11}}-Cu_{11}\nonumber \\
 & \geq & \frac{\sum_{i\geq2}(2u_{11}-\sigma_{2}^{ii})u_{11i}^{2}}{u_{11}^{2}}
 +\frac{1.5fu_{111}^{2}}{u_{11}^{3}}-\frac{\sigma_{2}^{11}u_{111}^{2}}{u_{11}^{2}}
 +\frac{f_{p_{i}}u_{i11}}{u_{11}}-Cu_{11}.\label{eq:b2}
\end{eqnarray}

Since $\sigma_3(\nabla^2 u)\ge -A$, by Corollary \ref{GQlem}, with $W=\nabla^2 u, \xi_{ij}=u_{ij1}$, as $u_{11}$ can be assumed to be sufficiently large at $x_0$,
\begin{eqnarray}
\frac{3fu_{111}^{2}}{2u_{11}^{3}}-\frac{\sigma_{2}^{11}u_{111}^{2}}{u_{11}^{2}}
\geq\frac{\sigma_{2}^{11}u_{111}^{2}}{10u_{11}^{2}}.\label{eq:acon1}
\end{eqnarray}

\begin{equation}
\sigma_{2}^{ii}b_{ii}\geq\frac{\sigma_{2}^{ii}b_{i}^{2}}{10}
+\frac{f_{p_{i}}u_{i11}}{u_{11}}-Cu_{11}.\label{eq:b3}
\end{equation}

Again, we may assume that at $x_0$, $b\ge 20$, by \eqref{eq:b3} and \eqref{eq:P1},
\begin{eqnarray*}
\sigma_{2}^{ij}P_{ij} & \geq & -4\frac{(n-1)\sigma_{1}}{\rho}-8\frac{\sigma_{2}^{ii}x_{i}^{2}}{\rho^{2}}+\alpha(f\sigma_{1}-3\sigma_{3})
+\frac{\sigma_{2}^{ii}b_{i}^{2}}{20b}\\
 &  & +\frac{f_{p_{i}}u_{i11}}{bu_{11}}+\alpha\tilde{f}_{k}u_{k}+\beta\tilde{f}_{k}x_{k}-C(|\beta|+|\alpha|+
 \frac{u_{11}}{b}).
\end{eqnarray*}
By the critical point condition \eqref{eq:fp},
\[
\frac{f_{p_{i}}u_{i11}}{bu_{11}}+\alpha\tilde{f}_{k}u_{k}+\beta\tilde{f}_{k}x_{k}=\sum_{i}\frac{4f_{p_{i}}x_{i}}{\rho}+C.
\]
We may assume that  $\rho^{2}b$ is sufficiently large at $x_{0}$,
\begin{eqnarray}
\sigma_{2}^{ij}P_{ij} & \geq & -4\frac{(n-1)\sigma_{1}}{\rho}-8\frac{\sigma_{2}^{ii}x_{i}^{2}}{\rho^{2}}+\alpha(f\sigma_{1}-3\sigma_{3})
+\frac{\sigma_{2}^{ii}b_{i}^{2}}{20b}\label{P2}\\
 &  & -C(|\beta|+|\alpha|+\frac{\sigma_{1}}{b}).\nonumber
\end{eqnarray}

\medskip{}

We divided it into three cases. Let us denote the coordinate of the maximum point as $x_0=(x_1,x_2,\cdots,x_n)$.

\medskip{}

\noindent \textit{Case 1}: $|x_{0}|^{2}\leq\frac{1}{2}$.

In this case, $\frac{1}{\rho}\leq2$. By Newton-MacLaurin inequality, we have
\begin{eqnarray*}
\sigma_{2}^{ij}P_{ij}\geq-8(n-1)\sigma_{1}-16(n-1)\sigma_{1}+\frac{\inf_{B_{1}}f}{n}\alpha\sigma_{1}
-C(|\beta|+|\alpha|+\frac{\sigma_{1}}{b}).
\end{eqnarray*}
Pick $\alpha=\frac{24n^{2}+C}{\inf_{B_{1}}f}$,
the estimate follows in this case.

\medskip{}

\noindent \textit{Case 2}: $|x_{0}|^{2}\geq\frac{1}{2}$ and there
exists $j\geq2$, such that $x_{j}^{2}\geq\frac{1}{2n}$.

By \eqref{eq:fp},
\begin{eqnarray*}
\frac{\sigma_{2}^{jj}b_{j}^{2}}{b^{2}} & = & \sigma_{2}^{jj}[-\frac{4x_{j}}{\rho}+\alpha u_{j}u_{jj}+\beta x_{j}u_{jj}]^{2}.
\end{eqnarray*}

In view of (\ref{eq:assumW}), we may assume that
\[|-\frac{4x_{j}}{\rho}+\alpha u_{j}u_{jj}+\beta x_{j}u_{jj}|\ge |\frac{2x_{j}}{\rho}|.\]
It follows from (\ref{s222}) that,
\begin{equation}
\frac{\sigma_{2}^{jj}b_{j}^{2}}{b^{2}}\text{\ensuremath{\geq}}
\sigma_{2}^{jj}(\frac{2x_{j}}{\rho})^{2}\geq\frac{2c_{2}\sigma_{1}}{n\rho^{2}}.
\end{equation}
We get 
\begin{eqnarray*}
\sigma_{2}^{ij}P_{ij} & \geq & -4\frac{(n-1)\sigma_{1}}{\rho}-8\frac{\sigma_{2}^{ii}x_{i}^{2}}{\rho^{2}}+\frac{c_{2}b\sigma_{1}}{10n\rho^{2}}
-C(|\beta|+|\alpha|+\frac{\sigma_{1}}{b})\\
 & \geq & -4\frac{(n-1)\sigma_{1}}{\rho}-8\frac{\sigma_{1}}{\rho^{2}}+\frac{c_{2}b\sigma_{1}}{10n\rho^{2}}--C(|\beta|+|\alpha|+\frac{\sigma_{1}}{b}).
\end{eqnarray*}
Thus, the estimate follows for \textit{Case 2}.

\medskip

\noindent \textit{Case 3}: $|x_{0}|^{2}\geq\frac{1}{2}$ and $x_{1}^{2}\geq\frac{1}{2n}$.

From \eqref{eq:fp},
\begin{eqnarray*}
0=-\frac{4x_{1}}{\rho}+(\alpha u_{1}+\beta x_{1})u_{11}+\frac{b_{1}}{b}.
\end{eqnarray*}

We may assume $\beta u_{11}\geq\frac{12}{\rho}$, otherwise we would have
the estimate. We choose $\beta$ such that
\begin{equation}
\frac{\beta}{\sqrt{2n}}=3\alpha\sup_{B_{1}}|\nabla u|.\label{gh}
\end{equation}
Then we have
\begin{equation}
(\frac{b_{1}}{b})^{2}\geq\frac{\beta^{2}x_{1}^{2}u_{11}^{2}}{9}\geq\frac{\beta^{2}u_{11}^{2}}{18n}.
\end{equation}

Together with (\ref{tildec0}),
\begin{eqnarray}
\sigma_{2}^{ij}P_{ij} & \geq & -4\frac{(n-1)\sigma_{1}}{\rho}-8\frac{\sigma_{2}^{ii}x_{i}^{2}}{\rho^{2}}+\frac{\sigma_{2}^{11}b_{1}^{2}}{20b}
-C(|\beta|+|\alpha|+\frac{\sigma_{1}}{b})\nonumber \\
 & \geq & -4\frac{(n-1)\sigma_{1}}{\rho}-8\frac{(n-1)\sigma_{1}}{\rho^{2}}+\frac{b\beta^{2}\sigma_{2}^{11}u_{11}^{2}}{360n}
 -C(|\beta|+|\alpha|+\frac{\sigma_{1}}{b})\nonumber \\
 & \geq & -4\frac{(n-1)\sigma_{1}}{\rho}-8\frac{(n-1)\sigma_{1}}{\rho^{2}}+\frac{c_{1}\alpha^{2}b\sigma_{1}}{360n}
 -C(|\beta|+|\alpha|+\frac{\sigma_{1}}{b}).
\end{eqnarray}
Hence the estimate follows in \textit{Case 3}. Proof of  Theorem \ref{thm1} is complete. \qed


\section{Scalar Curvature Equation}

We turn to the prescribing scalar curvature equation. Let $M\subset \mathbb R^{n+1}$ be a piece of hypersurface as a graph over $B_r\subset \mathbb R^{n+1}$ for some $r>0$, denote $\nu$ be the normal of the hypersurface and $\kappa_1, \cdots, \kappa_n$ be the principal curvatures of $M$. The scalar curvature of $M$ is $\sigma_2(\kappa_1,\cdots, \kappa_n)$, assume it satisfies equation (\ref{seq:maineq}) for some positive $C^2$ function $f$ in $\mathbb R^{n+1}\times \mathbb S^n$.
\begin{theorem}\label{thm1s} Suppose $M$ is a graph over $B_r\subset \mathbb R^n$ and it is
a solution of equation \eqref{seq:maineq}
and assume that there is a nonnegative constant
$A$ such that the second fundamental form  $h$ of $M$
\begin{equation}\label{saA} \sigma_{3}(\nabla^{2}h(x))\ge-A, \forall x\in B_r,\end{equation} then
\begin{equation}
\max_{x\in B_{\frac{r}{2}}}|\kappa_i(x)|\leq C,\label{sestH}
\end{equation}
where constant $C$ depending only on $n$, $r$, $A$, $\|M\|_{C^{1}(B_r)}$,
$\|f\|_{C^{2}(B_r)}$ and $\|\frac{1}{f}\|_{L^{\infty}}(B_r)$. \end{theorem}

Suppose that a hypersurface $M$ in $\mathbb{R}^{n+1}$ can be written
as a graph over $B_{r}\subseteq\mbox{\ensuremath{\mathbb{R}}}^{n}$.
At any point of $x\in B_{1}$, the principal curvature $\kappa=(\kappa_{1},\kappa_{2},\cdots,\kappa_{n})$
of the graph $M=(x,u(x))$ satisfy a equation
\begin{equation}
\sigma_{2}(\kappa)=f(X,\nu)>0,\label{eq:graph}
\end{equation}
where $X$ is the position vector of $M$, and $\nu$ a outer normal
vector on $M$.

We choose an orthonormal frame in $\mathbb{R}^{n+1}$ such that $\{e_{1},e_{2},\cdots,e_{n}\}$
are tangent to $M$ and $\nu$ is outer normal on $M$. We recall
the following fundamental formulas of a hypersurface in $\mathbb{R}^{n+1}$:
\begin{eqnarray*}
X_{ij} & = & -h_{ij}\nu\begin{array}{cc}
 & (Gauss\,\,formula)\end{array}\\
\nu_{i} & = & h_{ij}e_{j}\begin{array}{cc}
 & (Weingarten\,\,equation)\end{array}\\
h_{ijk} & = & h_{ikj}\begin{array}{cc}
 & (Codazzi\,\,equation)\end{array}\\
R_{ijkl} & = & h_{ik}h_{jl}-h_{il}h_{jk}\begin{array}{cc}
 & (Gauss\,\,equation)\end{array},
\end{eqnarray*}
where $R_{ijkl}$ is the curvature tensor. We also have the following
commutator formula:
\begin{eqnarray}
h_{ijkl}-h_{ijlk} & = & h_{im}R_{mjkl}+h_{mj}R_{mikl}.\label{eq:commute}
\end{eqnarray}

Combining Codazzi equation, Gauss equation and (\ref{eq:commute}),
we have
\begin{eqnarray}
h_{iikk} = h_{kkii}+\sum_{m}(h_{im}h_{mi}h_{kk}-h_{mk}^{2}h_{ii}).\label{eq:commute2}
\end{eqnarray}

We now prove Theorem \ref{thm1s} in this section. The idea is similar to the proof of Theorem \ref{thm1}, the construction of the test function is a little more subtle.

\medskip

\noindent
{\it Proof of Theorem \ref{thm1s}.}
For simplicity of notation, we work on the case $r=1$. The argument here can easily be carried over for general $r>0$. At any point $X(x)\in M$ and any unit tangential vector $\vartheta$ on
$M$, we consider the auxiliary function in $B_{1}$
\begin{equation}
P(X(x),\vartheta)=2\log\rho(X)+\log\log h_{\vartheta\vartheta}-\beta\frac{(X,\nu)}{(\nu,E_{n+1})}+\alpha\frac{1}{(\nu,E_{n+1})^{2}},
\end{equation}
where $E_{n+1}=(0,\cdots,0, 1)$, $\rho(X)=1-|X|_{\mathbb{R}^{n+1}}^{2}+(X,E_{n+1})^{2}=1-|x|_{\mathbb{R}^{n}}^{2}$,
$\alpha,\beta$ are constants to be determined later. So the maximum
of this function is attained in interior point of $B_{1}$, say $x_{0}$,
and $\vartheta(x_{0})=e_{1}(x_{0})$. It is easy to see $h_{1i}(x_{0})=0$,
for any $i=2,\cdots n$ after you fix $e_{1}$. Then rotate $\{e_{2},e_{3},\cdots e_{n}\}$
such that $h_{ij}(x_{0})$ is diagonal.
Denote $F^{ij}=\frac{\partial\sigma_{2}}{\partial h_{ij}}$,
which is positive definite when $\kappa\in\Gamma_{2}$. At point $x_{0}$,
\begin{equation}\label{eq:critical}
0=P_{i}=\frac{2\rho_{i}}{\rho}+\frac{h_{11i}}{h_{11}\log h_{11}}-\beta[\frac{(X,\nu)}{(\nu,E_{n+1})}]_{i}-\alpha\frac{2(\nu_{i},E_{n+1})}{(\nu,E_{n+1})^{3}},
\end{equation}
and
\begin{eqnarray}\label{eq:maximum}
0\geq F^{ij}P_{ij} & = & \frac{2F^{ii}\rho_{ii}}{\rho}-\frac{2F^{ii}\rho_{i}^{2}}{\rho^{2}}   -\beta F^{ii}[\frac{(X,\nu)}{(\nu,E_{n+1})}]_{ii} \\ \nonumber
 &  & +\frac{F^{ii}h_{11ii}}{h_{11}\log h_{11}}-(\log h_{11}+1)\frac{F^{ii}h_{11i}^{2}}{h_{11}^{2}\log^{2}h_{11}}\\ \nonumber
 &  & -\alpha\frac{2F^{ii}(\nu_{ii},E_{n+1})}{(\nu,E_{n+1})^{3}}+6\alpha\frac{F^{ii}(\nu_{i},E_{n+1})^{2}}{(\nu,E_{n+1})^{4}}.
\end{eqnarray}

From the fundamental formulas from hypersurface, we have
\begin{equation}
\nu_{ii}=(h_{ik}e_{k})_{i}=h_{iik}e_{k}-h_{ik}h_{ki}\nu,\label{eq:nu}
\end{equation}

\begin{equation}\label{eq:Xnu1}
[\frac{(X,\nu)}{(\nu,E_{n+1})}]_{i}=\sum_{l}\frac{(X,e_{l})h_{il}}{(\nu,E_{n+1})}-\frac{(e_{l},E_{n+1})h_{il}(X,\nu)}{(\nu,E_{n+1})^{2}},
\end{equation}

and
\begin{eqnarray}\label{eq:Xnu2}
[\frac{(X,\nu)}{(\nu,E_{n+1})}]_{ii} & = & \frac{h_{ii}}{(\nu,E_{n+1})}-\frac{\sum_{l}(X,\nu)h_{il}^{2}}{(\nu,E_{n+1})}+\frac{(X,e_{l})h_{ili}}{(\nu,E_{n+1})}\\ \nonumber
 &  & -\sum_{l,k}\frac{(e_{k,}E_{n+1})h_{ki}(X,e_{l})h_{il}}{(\nu,E_{n+1})^{2}}
 -\frac{h_{il}(e_{l,}E_{n+1})(X,e_{k})h_{ki}}{(\nu,E_{n+1})^{2}}\\ \nonumber
 &  & +\sum_{l}\frac{h_{li}^{2}(X,\nu)}{(\nu,E_{n+1})}-\frac{(e_{l},E_{n+1})(X,\nu)h_{ili}}{(\nu,E_{n+1})^{2}}\\ \nonumber
 &  & +\sum_{l,k}\frac{2(e_{l},E_{n+1})h_{il}(X,\nu)h_{ki}(e_{k},E_{n+1})}{(\nu,E_{n+1})^{3}}\\ \nonumber
 & = & \frac{h_{ii}}{(\nu,E_{n+1})}-2[\frac{(X,\nu)}{(\nu,E_{n+1})}]_{i}\frac{\sum_{l}(e_{l},E_{n+1})h_{il}}{(\nu,E_{n+1})}\\ \nonumber
 &  & +\frac{(X,e_{l})(\nu,E_{n+1})-(e_{l},E_{n+1})(X,\nu)}{(\nu,E_{n+1})^{2}}h_{iil}.
\end{eqnarray}
Moreover,
\begin{equation}\label{eq:rho1}
\rho_{i}=-2(X,e_{i})+2(X,E_{n+1})(e_{i},E_{n+1}),
\end{equation}
and
\begin{equation}\label{eq:rho2}
\rho_{ii}=-2+2(X,\nu)h_{ii}+2(e_{i},E_{n+1})^{2}-2(X,E_{n+1})h_{ii}(\nu,E_{n+1}).
\end{equation}

Insert(\ref{eq:critical}), (\ref{eq:nu}), (\ref{eq:Xnu1}), (\ref{eq:Xnu2}),
(\ref{eq:rho1}) and (\ref{eq:rho2}) into (\ref{eq:maximum}),
\begin{eqnarray}\label{eq:max2}
F^{ij}P_{ij} & = & \frac{4F^{ii}(-1+(e_{i},E_{n+1})^{2}+(X,\nu)h_{ii}-(X,E_{n+1})h_{ii}(\nu,E_{n+1}))}{\rho}\\ \nonumber
 &  & -\frac{8F^{ii}[(X,e_{i})-(X,E_{n+1})(e_{i},E_{n+1})]^{2}}{\rho^{2}}   -(\log h_{11}+1)\frac{F^{ii}h_{11i}^{2}}{h_{11}^{2}\log^{2}h_{11}}\\ \nonumber
 &  & +2F^{ii}\frac{\sum_{l}(e_{l},E_{n+1})h_{il}}{(\nu,E_{n+1})}(\frac{2\rho_{i}}{\rho}+\frac{h_{11i}}{h_{11}\log h_{11}}-\alpha\frac{2(\nu_{i},E_{n+1})}{(\nu,E_{n+1})^{3}})\\ \nonumber
 &  & -\beta\frac{F^{ii}h_{ii}}{(\nu,E_{n+1})}-\beta F^{ii}h_{iil}\frac{(X,e_{l})(\nu,E_{n+1})-(e_{l},E_{n+1})(X,\nu)}{(\nu,E_{n+1})^{2}}\\ \nonumber
 &  & -\alpha\frac{2F^{ii}(h_{iik}e_{k}-h_{ik}h_{ki}\nu,E_{n+1})}{(\nu,E_{n+1})^{3}}
 +6\alpha\frac{F^{ii}(\nu_{i},E_{n+1})^{2}}{(\nu,E_{n+1})^{4}} +\frac{F^{ii}h_{11ii}}{h_{11}\log h_{11}} \\ \nonumber
 & \geq & -\frac{4\sum_{i}F^{ii}}{\rho}+\frac{4F^{ii}h_{ii}[(X,\nu)-(X,E_{n+1})(\nu,E_{n+1})]}{\rho}+\frac{F^{ii}h_{11ii}}{h_{11}\log h_{11}}\\ \nonumber
 &  & -\frac{8F^{ii}[(X,e_{i})-(X,E_{n+1})(e_{i},E_{n+1})]^{2}}{\rho^{2}}-(\log h_{11}+1)\frac{F^{ii}h_{11i}^{2}}{h_{11}^{2}\log^{2}h_{11}}\\ \nonumber
 &  & +2F^{ii}\frac{(e_{i},E_{n+1})h_{ii}}{(\nu,E_{n+1})}(\frac{2\rho_{i}}{\rho}+\frac{h_{11i}}{h_{11}\log h_{11}}) -\alpha\frac{2F^{ii}h_{iik}(e_{k},E_{n+1})}{(\nu,E_{n+1})^{3}}\\ \nonumber
 &  & -\beta\frac{F^{ii}h_{ii}}{(\nu,E_{n+1})}-\beta F^{ii}h_{iil}\frac{(X,e_{l})(\nu,E_{n+1})-(e_{l},E_{n+1})(X,\nu)}{(\nu,E_{n+1})^{2}}+\frac{2\alpha F^{ii}h_{ii}^{2}}{(\nu,E_{n+1})^{2}}.
\end{eqnarray}

Differentiate equation (\ref{eq:graph}) twice, we
have
\begin{equation}\label{eq:G2}
F^{ij}h_{ijl}=d_{X}f(e_{l})+h_{kl}d_{\nu}f(e_{k}),
\end{equation}
and
\begin{eqnarray}\label{eq:G3}
F^{ii}h_{ii11} & = & \sum_{i\neq k}h_{ik1}^{2}-\sum_{i\ne k}h_{ii1}h_{kk1}\\ \nonumber
 &  & +d_{X}^{2}f(e_{1},e_{1})+h_{11}d_{X,\nu}^{2}f(e_{1},e_{1})\\ \nonumber
 &  & -h_{11}d_{X}f(\nu)+h_{k11}d_{\nu}f(e_{k})-h_{11}^{2}d_{\nu}f(\nu)\\ \nonumber
 &  & +h_{11}d_{\nu,X}^{2}f(e_{1},e_{1})+h_{11}^{2}d_{\nu}^{2}f(e_{1},e_{1}).
\end{eqnarray}

It follows from  (\ref{eq:G2}), (\ref{eq:G3}), (\ref{eq:commute2}), (\ref{eq:max2}) and (\ref{eq:critical})
\begin{eqnarray}
F^{ij}P_{ij} & \geq & \frac{-4\sum_{i}F^{ii}}{\rho}-C\frac{1}{\rho}-\frac{8F^{ii}[(X,e_{i})-(X,E_{n+1})(e_{i},E_{n+1})]^{2}}{\rho^{2}}\\ \nonumber
 &  & +\frac{\sum_{i\neq k}h_{ik1}^{2}-\sum_{i\ne k}h_{ii1}h_{kk1}+F^{ii}h_{11}^{2}h_{ii}-F^{ii}h_{ii}^{2}h_{11}}{h_{11}\log h_{11}}\\ \nonumber
 &  & -(\log h_{11}+1)\frac{F^{ii}h_{11i}^{2}}{h_{11}^{2}\log^{2}h_{11}}+\frac{h_{k11}d_{\nu}f(e_{k})}{h_{11}\log h_{11}}-C\frac{h_{11}}{\log h_{11}}\\ \nonumber
 &  & +2F^{ii}\frac{(e_{i},E_{n+1})h_{ii}}{(\nu,E_{n+1})}(\frac{2\rho_{i}}{\rho}+\frac{h_{11i}}{h_{11}\log h_{11}})-\beta d_{\nu}f(e_{k})[\frac{(X,\nu)}{(\nu,E_{n+1})}]_{k}\\ \nonumber
 &  & -C\beta -C\alpha-\alpha\frac{2h_{kl}d_{\nu}f(e_{l})(e_{k},E_{n+1})}{(\nu,E_{n+1})^{3}}
 +\alpha\frac{2F^{ii}h_{ii}^{2}}{(\nu,E_{n+1})^{2}}\\ \nonumber
 & \geq & \frac{-4\sum_{i}F^{ii}}{\rho}-C(\frac{1}{\rho}+\beta+\alpha)-\frac{8F^{ii}a_{i}^{2}}{\rho^{2}}\\ \nonumber
 &  & +\frac{\sum_{i\neq k}h_{ik1}^{2}-\sum_{i\ne k}h_{ii1}h_{kk1}-F^{ii}h_{ii}^{2}h_{11}}{h_{11}\log h_{11}}-C\frac{h_{11}}{\log h_{11}}+\alpha\frac{2F^{ii}h_{ii}^{2}}{(\nu,E_{n+1})^{2}}\\ \nonumber
 &  & -(\log h_{11}+1)\frac{F^{ii}h_{11i}^{2}}{h_{11}^{2}\log^{2}h_{11}}+2F^{ii}\frac{(e_{i},E_{n+1})h_{ii}}{(\nu,E_{n+1})}(\frac{2\rho_{i}}{\rho}+\frac{h_{11i}}{h_{11}\log h_{11}}).
\end{eqnarray}
where $a_{i}=(X,e_{i})-(X,E_{n+1})(e_{i},E_{n+1})$
and $C$ is a constant depend only on $||f||_{C^{2}},||u||_{C^{1}}$.
In the rest of this article, we will denote $C$ to be constant under
control (depending only on $n$, $\|f\|_{C^{2}}$, $\|\frac{1}{f}\|_{L^{\infty}}$
and $\|u\|_{C^{1}}$), which may change line by line.

By Cauchy-Schwarz inequality
\begin{equation}\label{eq:CS1}
4F^{ii}\frac{|h_{ii}\rho_{i}|}{\rho(\nu,E_{n+1})}\geq-\frac{2F^{ii}h_{ii}^{2}}{(\nu,E_{n+1})^{2}}-\frac{2F^{ii}\rho_{i}^{2}}{\rho^{2}},
\end{equation}
and
\begin{equation}\label{eq:CS2}
2F^{ii}\frac{|h_{11i}h_{ii}|}{(\nu,E_{n+1})h_{11}\log h_{11}}\geq-\frac{F^{ii}h_{ii}^{2}}{(\nu,E_{n+1})^{2}}-\frac{F^{ii}h_{11i}^{2}}{h_{11}^{2}\log^{2}h_{11}}.
\end{equation}

By (\ref{eq:CS1}) and (\ref{eq:CS2}), we have
\begin{eqnarray*}
F^{ij}P_{ij} & \geq & \frac{-4\sum_{i}F^{ii}}{\rho}-C(\frac{1}{\rho}+\beta+\alpha)-\frac{16F^{ii}a_{i}^{2}}{\rho^{2}}\\
 &  & +\frac{\sum_{i\neq j}h_{ij1}^{2}-\sum_{i\ne j}h_{ii1}h_{jj1}-F^{ii}h_{ii}^{2}h_{11}}{h_{11}\log h_{11}}-C\frac{h_{11}}{\log h_{11}}\\
 &  & -(\log h_{11}+2)\frac{F^{ii}h_{11i}^{2}}{h_{11}^{2}\log^{2}h_{11}}+(2\alpha-3)\frac{F^{ii}h_{ii}^{2}}{(\nu,E_{n+1})^{2}}.
\end{eqnarray*}

By Corollary \ref{GQlemW}, for any $j$ from $2$ to $n$,
\begin{equation}\label{eqs:assum}
|h_{jj}|\leq Ch^{-\frac12}_{11}.
\end{equation}
 As $h_{11}$ is sufficiently large at $x_{0}$, it
follows from Corollary \ref{GQlemW} that
\begin{equation}\label{eq:concave}
\frac{\sum_{i\neq j}h_{ij1}^{2}-\sum_{i\ne j}h_{ii1}h_{jj1}}{h_{11}\log h_{11}}-(\log h_{11}+2)\frac{F^{ii}h_{11i}^2}{h_{11}^2\log^2 h_{11}}\geq\frac{F^{ii}h_{11i}^{2}}{10h_{11}^{2}\log h_{11}}-C\frac{h_{11}}{\log h_{11}}.
\end{equation}

From (\ref{eq:concave}),
\begin{eqnarray*}
F^{ij}P_{ij} & \geq & \frac{F^{ii}h_{11i}^{2}}{20h_{11}^{2}\log h_{11}}-\frac{4\sum_{i}F^{ii}}{\rho}-C(\frac{1}{\rho}+\beta+\alpha)-\frac{16F^{ii}a_{i}^{2}}{\rho^{2}}\\
 &  & +(2\alpha-4)\frac{F^{ii}h_{ii}^{2}}{(\nu,E_{n+1})^{2}}.
\end{eqnarray*}

If we choose $h_{11}\geq C(\frac{1}{\rho}+\beta+\alpha)$ and $\alpha>4$,
\begin{equation}
F^{ij}P_{ij}\geq\frac{F^{ii}h_{11i}^{2}}{20h_{11}^{2}\log h_{11}}-\frac{4\sum_{i}F^{ii}}{\rho}-\frac{16F^{ii}a_{i}^{2}}{\rho^{2}}+\alpha\frac{F^{ii}h_{ii}^{2}}{(\nu,E_{n+1})^{2}}.
\end{equation}

So far all computations above are by the local frame on
hypersurface, we now switch to orthonormal coordinate $\{E_{1},E_{2},\cdots,E_{n},E_{n+1}\}$
in $\mathbb{R}^{n+1}$, such that $E_{i}\perp E_{n+1}$.
In this new coordinates, we can decompose vector $X$ as follow
\begin{equation}
X=\sum_{i=1}^{n}(X,E_{i})E_{i}+(X,E_{n+1})E_{n+1}.
\end{equation}
Thus
\begin{equation}
\rho=1-\sum_{i=1}^{n}(X,E_{i})^{2}.
\end{equation}

Recall that
\begin{equation}\label{eq:aj}
a_{j}=(X,e_{i})-(X,E_{n+1})(e_{i},E_{n+1})=\sum_{i}(X,E_{i})(E_{i},e_{j}).
\end{equation}

So
\begin{eqnarray*}
\sum_{i=1}^{n}a_{i}^{2} & = & \sum_{i,k,l=1}^{n}(X,E_{k})(X,E_{l})(E_{k},e_{i})(E_{l},e_{i})\\
 & = & \sum_{k,l=1}^{n}(X,E_{k})(X,E_{l})(\delta_{kl}-(E_{k},\nu)(E_{l},\nu)).
\end{eqnarray*}
Because $(\nu,E_{n+1})$ is bounded from above and below,
we have
\begin{equation}\label{eq:equvalent}
\frac{1}{C}\sum_{i=1}^{n}(X,E_{i})^{2}\leq\sum_{i=1}^{n}a_{i}^{2}\leq C\sum_{i=1}^{n}(X,E_{i})^{2}.
\end{equation}

We divided it into three cases.

\medskip{}

\noindent \textit{Case 1}: $\sum_{i}|(X,E_{i})|^{2}\leq\frac{1}{2}$.

In this case, $\frac{1}{\rho}\leq2$. If $\alpha$ is chosen sufficiently large,
we get the estimate.

Therefore,  we may assume that $\sum_{i}|(X,E_{i})|^{2}\geq\frac{1}{2}$,
which is equivalent to $\sum_{i=1}^{n}a_{i}^{2}\geq\frac{1}{2C}$.

\noindent \textit{Case 2}: for some $n\geq j>1$, $|a_{j}|>d$, where
$d$ is a small positive constant to be determined. From (\ref{eq:critical}),
\begin{eqnarray}\label{AJ}
\frac{4a_{j}}{\rho} & = & \frac{h_{11j}}{h_{11}\log h_{11}}-\beta[\frac{(X,\nu)}{(\nu,E_{n+1})}]_{j}-\alpha\frac{2(\nu_{j},E_{n+1})}{(\nu,E_{n+1})^{3}}\\ \nonumber
 & = & \frac{h_{11j}}{h_{11}\log h_{11}}-\beta\frac{b_{j}h_{jj}}{(\nu,E_{n+1})^{2}}-\alpha\frac{2h_{jj}(e_{j},E_{n+1})}{(\nu,E_{n+1})^{3}},
\end{eqnarray}
where $b_{j}=(X,e_{j})(\nu,E_{n+1})-(e_{j},E_{n+1})(X,\nu)$.

By the assumption and Corollary \ref{GQlemW}, $h_{ii}, i=2,\cdots,n$ are small, as we may assume $a(\beta+\alpha)\leq\epsilon d$, where $a^2=\frac{A}{\sigma_1(h)}$. It follows from (\ref{AJ}),
\begin{equation}
\frac{h_{11j}^{2}}{h_{11}^{2}\log^{2}h_{11}}\geq\frac{4a_{j}^{2}}{\rho^{2}}\geq\frac{4d^{2}}{\rho^{2}}.
\end{equation}
In turn,
\begin{eqnarray}
F^{ij}P_{ij} & \geq & \frac{F^{jj}d^{2}\log h_{11}}{5\rho^{2}}-\frac{4\sum_{i}F^{ii}}{\rho}-\frac{16F^{ii}a_{i}^{2}}{\rho^{2}}.
\end{eqnarray}

By Lemma \ref{lemma1}
\begin{equation}
F^{jj}\geq c\sum_{i}F^{ii},
\end{equation}
the estimate follows in this case.

\noindent \textit{Case 3}: $|x_{0}|^{2}\geq\frac{1}{2}$ and $\sum_{i\neq1}|a_{i}|^{2}\leq(n-1)d^{2}$
.

\noindent In this case we have from (\ref{eq:equvalent}) that
\begin{equation}
|a_{1}|\geq\frac{1}{4C}.
\end{equation}

\noindent At the critical point
\begin{equation}
\frac{h_{111}}{h_{11}\log h_{11}}=\frac{4a_{1}}{\rho}+\beta\frac{b_{1}h_{11}}{(\nu,E_{n+1})^{2}}+\alpha\frac{2h_{11}(e_{1},E_{n+1})}{(\nu,E_{n+1})^{3}},
\end{equation}
where $b_{1}=(X,e_{1})(\nu,E_{n+1})-(E_{n+1},e_{1})(X,\nu)$.

We claim that $|b_{1}|$ have a positive lower bound, say $|b_{1}|\geq c_{0}>0$.
In fact, we can decompose vector $\nu$ as follow
\begin{equation}
\nu=\sum_{i}(E_{i},\nu)E_{i}+(E_{n+1},\nu)E_{n+1}.
\end{equation}
and
\begin{eqnarray*}
b_{1} & = & \sum_{i}(X,E_{i})(E_{i},e_{1})(\nu,E_{n+1})-(E_{n+1},e_{1})(X,E_{i})(E_{i},\nu)\\
 &  & +(E_{n+1},e_{1})[(X,E_{n+1})(\nu,E_{n+1})-(X,E_{n+1})(E_{n+1},\nu)]\\
 & = & a_{1}(\nu,E_{n+1})-(E_{n+1},e_{1})(X,E_{i})(E_{i},\nu).
\end{eqnarray*}

Note that
\begin{equation}
e_{1}-\sum_{j}(e_{1},E_{j})E_{j}=(e_{1},E_{n+1})E_{n+1},\label{eq:e3}
\end{equation}
and
\begin{equation}\label{eq:epsi}
E_{i}-\sum_{j}(e_{j},E_{i})e_{j}=(E_{i},\nu)\nu.
\end{equation}

Take inner product of (\ref{eq:epsi}) with (\ref{eq:e3}),
\begin{equation}\label{eq:epsi2}
\sum_{j,k}(e_{1},E_{j})(E_{j},e_{k})(e_{k},E_{i})-(e_{1},E_{i})=(E_{n+1},e_{1})(E_{i},\nu)(E_{n+1},\nu)
\end{equation}
Recall that
\begin{equation}\label{eq:a1}
a_{k}=\sum_{k}(X,E_{i})(E_{i},e_{k}).
\end{equation}
Then by (\ref{eq:a1}) and (\ref{eq:epsi2}), $b_{1}$ becomes
\begin{eqnarray*}
b_{1} & = & a_{1}(\nu,E_{n+1})+\frac{-\sum_{k,j}a_{k}(e_{1},E_{j})(E_{j},e_{k})+a_{1}}{(E_{n+1},\nu)}\\
 & = & a_{1}(\nu,E_{n+1})+a_{1}\frac{1-\sum_{j}(e_{1},E_{j})^{2}}{(E_{n+1},\nu)}
 -\sum_{k\neq1}\sum_{j}\frac{a_{k}(e_{1},E_{j})(E_{j},e_{k})}{(E_{n+1},\nu)}.
\end{eqnarray*}

By our assumption in this case,
\begin{eqnarray*}
|b_{1}| & \geq & |a_{1}(\nu,E_{n+1})|-\sum_{k\neq1}\frac{|a_{k}|}{|(E_{n+1},\nu)|}
\end{eqnarray*}
We may choose $d$ small. Then there is positive lower bound for $|b_{1}|$
, such that
\begin{equation}
|b_{1}|\geq\frac{1}{2}|a_{1}||(\nu,E_{n+1})|.
\end{equation}

If we choose $\beta\geq\alpha$ large,  we get
\begin{eqnarray*}
\frac{|h_{111}|}{h_{11}\log h_{11}} & \geq & -|\frac{4a_{1}}{\rho}|+\beta|\frac{a_{1}h_{11}}{4(\nu,E_{n+1})}|\\
 & \geq & \beta\frac{|a_{1}|h_{11}}{8|(\nu,E_{n+1})|}.
\end{eqnarray*}

Thus,
\begin{equation}
F^{ij}P_{ij}\geq\frac{a_{1}^{2}\beta^{2}F^{11}h_{11}^{2}\log h_{11}}{20(\nu,E_{n+1})^{2}}-\frac{4\sum_{i}F^{ii}}{\rho}-\frac{16F^{ii}a_{i}^{2}}{\rho^{2}}.
\end{equation}
It follows from Lemma \ref{lemma1},
\begin{equation}
F^{11}h_{11}^{2}\geq C\sum_{i}F^{ii},
\end{equation}
for some dimensional constant $C>0$.
This implies the interior curvature estimate. The proof of Theorem \ref{thm1s} is complete. \qed

\bigskip

The proof of Theorem \ref{thm01} follows the same lines of proof Theorem \ref{thm1s} as the isometrically embedded hypersurface obeys the same curvature equation (\ref{seq:maineq}). We may use Corollary \ref{GQlemR} in place of Corollary \ref{GQlemW} in the proof of Theorem \ref{thm1s}. The control of Ricci curvature by intrinsic metric implies condition $\sigma_3(\kappa)$ bounded from below as in Corollary \ref{GQlemR}.

\end{document}